\documentclass[12 pt]{amsart}
\usepackage{latexsym, amsmath, amssymb}
\usepackage[english]{babel}
\usepackage[utf8x]{inputenc}
\usepackage{hyperref}
\usepackage{todonotes}

\usepackage{graphicx}
\usepackage{cleveref}

\usepackage{tikz-cd}

\usepackage{stackengine}
\usepackage{xcolor}

\newcommand\I{{\mathbb{I}}}

\newcommand{\cS}{\mathcal{S}}

\newcommand\cD{{\mathcal{D}}}

\newcommand{\Q}{\mathbb{Q}}

\newcommand{\Z}{\mathbb{Z}}

\newcommand\Gal{{\mathrm {Gal}}}

\usepackage[OT2,T1]{fontenc}
\DeclareSymbolFont{cyrletters}{OT2}{wncyr}{m}{n}
\DeclareMathSymbol{\Sha}{\mathalpha}{cyrletters}{"58}

\newtheorem{thm}{Theorem}
\newtheorem{theorem}[thm]{Theorem}
\newtheorem{cor}[thm]{Corollary}

\newtheorem{prop}[thm]{Proposition}
\newtheorem{lemma}[thm]{Lemma}
\newtheorem*{theoremquartic}{Theorem A}
\theoremstyle{definition}
\newtheorem{definition}{Definition}
\theoremstyle{remark}
\newtheorem{remark}{Remark}
\DeclareMathOperator{\Cl}{C\ell}

\newcommand\remove[1]{}

\title[Euclidean Algorithms for Ideal Classes in Biquadratic fields]{Euclidean Algorithms for Ideal Classes in Biquadratic fields: A Genus-Theoretic Perspective }

\author{Sunil Kumar Pasupulati}

\begin{document}
	\maketitle

\begin{abstract}
We study Euclidean ideal classes in real biquadratic fields and obtain unconditional existence results via genus theory. Lenstra showed (assuming the Generalized Riemann Hypothesis) that a number field with unit rank at least one admits a Euclidean ideal precisely when its class group is cyclic; subsequent work has aimed to remove the GRH hypothesis in special families. Focusing on real biquadratic fields  $K=\mathbb{Q}\left(\sqrt{d_1},\sqrt{d_2}\right)$
with $2\nmid d_1d_2$, we prove that if the class group $\mathrm{Cl}_K$ is cyclic and the Hilbert class field $H(K)$ is abelian over $\mathbb{Q}$, then $K$ contains a Euclidean ideal class (unconditionally). We also analyse the distribution of genus numbers in a natural family of biquadratic fields and, using these statistics, show that the set of biquadratic fields admitting a Euclidean ideal has density zero.
\end{abstract}

\section{Introduction}
The study of \emph{Euclidean ideals} in number fields emerges as a natural extension of the classical theory of Euclidean domains. This concept was introduced by Lenstra\cite{Len79} in 1979 as a framework for understanding the algorithmic structure of Dedekind domains with cyclic class groups. A \emph{Euclidean ideal class} enables the reduction of ideals within a particular ideal class algorithmically, even when the ring of integers is not Euclidean in the traditional sense. Lenstra demonstrated that, under the assumption of the Generalised Riemann Hypothesis (GRH), a number field with unit rank at least one and cyclic class group possesses a Euclidean ideal class.

Earlier, Weinberger \cite{Wei72} had also shown under GRH that the only number fields with unit rank at least one whose ring of integers is a Euclidean domain are precisely those with a trivial class group. Lenstra extended this perspective by shifting the focus from the Euclidean property of the ring itself to that of its ideal classes, thereby identifying the cyclicity of the class group as the central criterion for the existence of a Euclidean algorithm in this broader context.

Recently, there has been significant progress in establishing the existence of Euclidean ideals \emph{unconditionally}. Graves and Murty~\cite{GrMu13} removed the dependence on GRH for number fields with Abelian Hilbert class field and unit rank at least four, using a \emph{growth lemma}~\cite{Gra13}. Deshouillers, Gun, and Sivaraman~\cite{DGS20} extended this to number fields with unit rank at least three, under the additional assumption that $\Gal\left(\mathbb{Q}(\zeta_f)/K\right)$ is cyclic, where $f$ is the smallest positive integer such that $K \subset \mathbb{Q}(\zeta_f)$. Gun and Sivaraman also introduced new criteria for the existence of Euclidean ideal classes in quadratic and cubic fields with cyclic class groups, using sieve methods and density arguments, and proved a sequential growth lemma~\cite[Theorem~14]{GuSi20}. For a broader perspective on Euclidean ideals and their connections to Artin's primitive root conjecture, see the elegant survey by Moree\cite{Mor12a}.

Among families of number fields with unit rank three, \emph{real biquadratic fields}, those of the form $K = \mathbb{Q}(\sqrt{d_1}, \sqrt{d_2})$ with $d_1, d_2 > 0$ occupy a prominent place due to their rich arithmetic structure and relative tractability. However, many such fields fail to satisfy the cyclicity condition on $\Gal(\mathbb{Q}(\zeta_f)/K)$, as documented in~\cite{KrPa23, KrPa25}. Several authors have investigated the existence of Euclidean ideals in real biquadratic fields with \emph{class number two}, particularly those of the form $K = \mathbb{Q}(\sqrt{q}, \sqrt{rs})$, where $q, r, s$ are distinct odd primes.

Graves~\cite{Gra11} demonstrated that the field $\mathbb{Q}(\sqrt{2}, \sqrt{35})$ admits a non-principal Euclidean ideal. Building on this, Hsu~\cite{Hsu16} generalized  Graves’s method to prove the existence of Euclidean ideals in fields of the form $\mathbb{Q}(\sqrt{q}, \sqrt{rs})$ with $q, r, s \equiv 1 \pmod{4}$. Chattopadhyay and Muthukrishnan \cite{ChMu19} further explored such fields under the assumption that $r, s \equiv 1 \pmod{4}$. More recently, the author, jointly with Krishnamoorthy~\cite{KrPa23, KrPa25}, established the existence of Euclidean ideal classes in fields $\mathbb{Q}(\sqrt{q}, \sqrt{rs})$ of class number two, except for  the case $q \equiv 1 \pmod{4}$ and $r, s \equiv 3 \pmod{4}$.

This line of work was influenced by Gun and Sivaraman~\cite{GuSi20}, who introduced a Galois-theoretic criterion for the existence of Euclidean ideals based on the structure of $\Gal(\mathbb{Q}(\zeta_f)/K)$. Their method involves ensuring this Galois group is not contained in the union of certain subgroups associated with the prime divisors of the conductor and the Hilbert class field. This approach reframes the problem in terms of explicit splitting conditions and Galois action, thus avoiding reliance on the Generalized Riemann Hypothesis.

In particular, it has been shown that for fields of the form $K = \mathbb{Q}(\sqrt{q}, \sqrt{rs})$ with $q, r, s$ distinct odd primes, a Euclidean ideal class exists unless $q \equiv 1 \pmod{4}$ and $r, s \equiv 3 \pmod{4}$. These exceptional cases correspond to specific Galois-theoretic obstructions related to embeddings into cyclotomic fields and the structure of the Hilbert class field.

In this article, we focus on the genus field and Hilbert class field of real biquadratic fields, and establish the following result:

\begin{theorem} \label{genustheory_for_biquadratic}
Let $K = \mathbb{Q}(\sqrt{d_1}, \sqrt{d_2})$ be a biquadratic field, where $d_i$'s are square free.
\[
d_1 = p_1 \cdots p_{t+s}, \quad d_2 = q_1 \cdots q_{\ell+r},
\]
with:
\begin{itemize}
  \item $p_i \equiv 1 \pmod{4}$ for $1 \leq i \leq t$,
  \item $p_{t+j} \equiv 3 \pmod{4}$ for $1 \leq j \leq s$,
  \item $q_i \equiv 1 \pmod{4}$ for $1 \leq i \leq \ell$,
  \item $q_{\ell+j} \equiv 3 \pmod{4}$ for $1 \leq j \leq r$.
\end{itemize}
 If $2\nmid d_1d_2$, then the Genus class field $G(K)$ of $K$ is given by:
\begin{equation}
\begin{aligned}
G(K) =
\begin{cases}
\mathbb{Q}\big( \sqrt{p_1}, \dots, \sqrt{p_t}, \sqrt{-p_{t+1}}, \dots, \sqrt{-p_{t+s}}, \\
\quad \sqrt{q_1}, \dots, \sqrt{q_\ell}, \sqrt{-q_{\ell+1}}, \dots, \sqrt{-q_{\ell+r}} \big), & \text{if } d_1 \equiv d_2 \equiv 1 \pmod{4}, \\[1ex]
\mathbb{Q}\big( \sqrt{p_1}, \dots, \sqrt{p_t}, \sqrt{-p_{t+1}}, \dots, \sqrt{-p_{t+s}}, \\
\quad \sqrt{q_1}, \dots, \sqrt{q_\ell}, \sqrt{-q_{\ell+1}}, \dots, \sqrt{-q_{\ell+r}}, \sqrt{-1} \big), & \text{otherwise}.
\end{cases}
\end{aligned}
\end{equation}
\end{theorem}

\begin{cor}
    Let $K = \mathbb{Q}(\sqrt{d_1}, \sqrt{d_2})$ be a biquadratic field with  $2\nmid d_1d_2$, then the Genus number of $K$ is $2^{\omega(\cD_K)-2}$.
\end{cor}
Utilising this explicit description of the genus field, we compute the Hilbert class field and leverage this structure to prove the existence of Euclidean ideals in biquadratic fields whose Hilbert class field is Abelian over $\mathbb{Q}$. Our main result is as follows:

\begin{theorem}
Let $K = \mathbb{Q}(\sqrt{d_1}, \sqrt{d_2})$ be a biquadratic number field with $2 \nmid d_1 d_2$, and suppose that the Hilbert class field $H(K)$ is Abelian over $\mathbb{Q}$. Then the ideal class group $C\ell_K$ is cyclic if and only if $K$ admits a Euclidean ideal class.
\end{theorem}

Let  
\[
\mathcal{S} = \left\{\, K = \mathbb{Q}(\sqrt{a}, \sqrt{b}) \;\middle|\; a,b \ \text{are odd and squarefree} \,\right\}
\]
denote the family of biquadratic fields 
We investigate the distribution of genus numbers among fields in $ \mathcal{S} $ using a parametrization of biquadratic fields due to Rome\cite{Rom18}.  
As an application of these results, we establish that the subset of $ \mathcal{S} $ consisting of biquadratic fields admitting a Euclidean ideal has density zero.

\begin{theorem}\label{density_Euclidean}
Let 
$\mathcal S_{\mathcal E}:=\{K\in\mathcal S:\ K\ \text{admits a Euclidean ideal} \ \}$.
	For $X>0$, define  $\mathcal S_{\mathcal E}(X):=\{K\in\mathcal S_{\mathcal E}:\ |\Delta_K|\le X\}$. Then 
	\[
	\lim_{X\to\infty}\frac{\#\mathcal S_{\mathcal E}(X)}{\#\mathcal S(X)}=0.
	\]
\end{theorem}

The paper is organized as follows. In Section~\ref{sec:prelim}, we present necessary preliminaries. Section~\ref{sec:structure} contains structural observations about the family $\mathbb{Q}(\sqrt{q}, \sqrt{rs})$. In Section~\ref{sec:genus}, we discuss genus theory and the Hilbert class field of biquadratic fields. Finally, in Section~\ref{sec:euclidean}, we establish the existence of Euclidean ideals in biquadratic fields with Abelian Hilbert class fields. In the last section, we study the distribution of genus numbers.

	\section{Preliminaries}\label{sec:prelim}
	\begin{lemma}
Let $K = \mathbb{Q}(\sqrt{m})$ be a quadratic field. Then the discriminant $D_K$ is given by:
\[
D_K = 
\begin{cases} 
m, & \text{if } m \equiv 1 \pmod{4}, \\
4m, & \text{if } m \equiv 2 \text{ or } 3 \pmod{4}.
\end{cases}
\]
\end{lemma}

\begin{definition}
Let $K$ be an algebraic number field. The \textbf{genus field} $G(K)$ of $K$ is the maximal abelian extension of $K$ satisfying the following:
\begin{enumerate}
\item $G(K) = k^*K$, where $k^*$ is an abelian number field,
\item $G(K)$ is unramified at all finite primes of $K$.
\end{enumerate}
The \textbf{genus number} $g_K$ is the degree $[G(K) : K]$, and the Galois group $\Gal(G(K)/K)$ is called the \textbf{genus group} of $K$.
\end{definition}

\begin{theorem}[Classical Genus Theory for Quadratic Fields] \label{genusnumber:quadratic}
Let $K = \mathbb{Q}(\sqrt{m})$, where $m$ is square-free, and let $D_K$ be its discriminant. Let $p_1, \dots, p_n$ be the distinct prime divisors of $D_K$. Define:
\[
p_i^* = 
\begin{cases}
(-1)^{\frac{p_i - 1}{2}} p_i, & \text{if } p_i \text{ is odd}, \\
\text{(as below)}, & \text{if } p_i = 2.
\end{cases}
\]
If $2$ divides $D_K$, then:
\[
2^* = 
\begin{cases}
-4, & \text{if } m \equiv 2 \pmod{4}, \\
2, & \text{if } m \equiv 2 \pmod{8}, \\
-2, & \text{if } m \equiv -2 \pmod{8}.
\end{cases}
\]
Then the genus field of $K$ is given by:
\[
G(K) = \mathbb{Q}\left( \sqrt{p_1^*}, \dots, \sqrt{p_n^*} \right).
\]
\end{theorem}

\begin{cor} \label{HK:Quadratic}
Let $K = \mathbb{Q}(\sqrt{d})$ be a real quadratic field, where:
\[
d = 2^g p_1 \cdots p_t p_{t+1} \cdots p_{t+s},
\]
with:
\begin{itemize}
\item $g = 0$ or $1$,
\item $p_i \equiv 1 \pmod{4}$ for $1 \leq i \leq t$,
\item $p_{t+j} \equiv 3 \pmod{4}$ for $1 \leq j \leq s$.
\end{itemize}
If $H(K)/\Q$ is abelian, then the Hilbert class field $H(K)$ is:
\[
H(K) = 
\begin{cases}
\mathbb{Q}\left( \sqrt{p_1}, \dots, \sqrt{p_t}, \sqrt{p_{t+1}p_{t+2}}, \dots, \sqrt{p_{t+1}p_{t+s}} \right), & \text{if } d \equiv 1 \pmod{4}, \\[0.5em]
\mathbb{Q}\left( \sqrt{p_1}, \dots, \sqrt{p_t}, \sqrt{p_{t+1}}, \dots, \sqrt{p_{t+s}} \right), & \text{if } d \equiv 3 \pmod{4}, \\[0.5em]
\mathbb{Q}\left( \sqrt{p_1}, \dots, \sqrt{p_t}, \sqrt{2p_{t+1}}, \dots, \sqrt{2p_{t+s}} \right), & \text{if } d \equiv -2 \pmod{8}, \\[0.5em]
\mathbb{Q}\left( \sqrt{p_1}, \dots, \sqrt{p_t}, \sqrt{p_{t+1}p_{t+2}}, \dots, \sqrt{p_{t+1}p_{t+s}}, \sqrt{2} \right), & \text{if } d \equiv 2 \pmod{8}.
\end{cases}
\]
\end{cor}

\begin{lemma}\label{Galiodgroup:Hilbert}
Let $K$ be a Galois number field. Then the extension $H(K)/\mathbb{Q}$ is Galois.
\end{lemma}

\begin{proof}
Let $H(K)$ denote the Hilbert class field of $K$, the maximal unramified abelian extension of $K$. From the class field theory:
\begin{itemize}
\item $H(K)/K$ is Galois, abelian, and unramified at all primes,
\item $H(K)$ is normal over $K$.
\end{itemize}
To show that $H(K)/\mathbb{Q}$ is Galois, we only need to show that $H(K)/\mathbb{Q}$ is normal.
Let $\sigma \in \operatorname{Aut}(\overline{\mathbb{Q}}/\mathbb{Q})$. Since $K/\mathbb{Q}$ is Galois, we have $\sigma(K) = K$. Because unramifiedness and abelianity are preserved under automorphisms, the conjugate field $\sigma(H(K))$ is also an unramified abelian extension of $K$. By the maximality of $H(K)$, we have $\sigma(H(K)) \subseteq H(K)$, hence  $\sigma(H(K)) = H(K)$. Thus, $H(K)$ is fixed under all automorphisms fixing $\mathbb{Q}$, so it is normal over $\mathbb{Q}$.
Hence, $H(K)/\mathbb{Q}$ is both separable and normal, and therefore Galois.
\end{proof}

The notion of Euclidean ideal is a generalisation of Euclidean domain defined by Lenstra\cite{Len79}.  \begin{definition} [Lenstra \cite{Len79}] \label{EI}	Let $R$ be a Dedekind domain and $\I$ be the set of non-zero integral ideals of $R$. The ideal  $C \in \I$   is called a  {\it Euclidean ideal} if there exists a function $\Psi:\I \to W$, where $W$ is  a well-ordered set, such that for every $I \in \I $ and  $x\in I^{-1}C \setminus C$,  there exists a  $y\in C$ such that 
		\begin{align*}
			\Psi\left( (x-y)IC^{-1} \right) < \Psi (I).
		\end{align*}
		We say  $\Psi$ is a Euclidean map for $C$.
		If $C$ is a Euclidean ideal, then every ideal in the ideal class $[C]$ is also a Euclidean ideal, and the ideal class $[C]$ is called  a {\it Euclidean ideal class.}
	\end{definition}

\noindent
The following  lemma provides the analytic foundation for the counting argument that follows:
It captures the asymptotic behaviour of squarefree integers with a prescribed number
of prime factors, and serves as a key input when estimating the number of biquadratic
fields with a given genus number. In essence, it is a version of the classical
Sathe--Selberg theorem adapted to the multiplicative structure appearing in our
parametrization.

    \begin{lemma}[Sathe--Selberg, \cite{Ten15book}]\label{lem:ss}
	Fix an integer $n\ge 1$. As $N\to\infty$,
	\[
	\#\{m\le N:\ \mu^2(m)=1,\ \omega(m)=n\}
	= \frac{N}{\log N}\,\frac{(\log\log N)^{\,n-1}}{(n-1)!}\,(1+o(1)).
	\]
\end{lemma}

    	\section{Genus theory for Biquadratic fields}\label{sec:genus}
	\begin{lemma}\label{ramification}
		Let $K=\Q(\sqrt{d_1}, \sqrt{d_2})$ be a biquadratic field. Then for any odd prime number $p$, the ramification index $e_p$ in $K/\mathbb{Q}$ satisfies $e_p \leq 2$.
	\end{lemma}
	
	\begin{proof}
		Let $K = \mathbb{Q}(\sqrt{d_1}, \sqrt{d_2})$ be a biquadratic field, i.e., a Galois extension of $\mathbb{Q}$ of degree $4$ with Galois group isomorphic to $(\mathbb{Z}/2\mathbb{Z})^2$. Let $p > 2$ be a prime number. Then $p$ may ramify in at most two of the three quadratic subfields of $K$, namely $\mathbb{Q}(\sqrt{d_1})$, $\mathbb{Q}(\sqrt{d_2})$, and $\mathbb{Q}(\sqrt{d_1d_2})$.
		Since the ramification index in each quadratic subfield is at most $2$, by the Abhyankar Lemma, if $p$ is tamely ramified in two linearly disjoint extensions with ramification index $2$, then the ramification index in their compositum is still $2$. Therefore, $e_p \leq 2$ for all odd primes $p$. Thus, in all cases, $e_p \leq 2$.
	\end{proof}

\begin{remark}
    	For $p = 2$, wild ramification may occur, but since $K$ is a multiquadratic extension of degree $4$, even in this case, the ramification index is at most $2$ in each quadratic subextension, and if one of $d_1$ or $d_2$ congruent $1\pmod{4}$  or  $2 \nmid d_1d_2$ then  $e_2$ is also at most $2$ in $K$ as well.
		
\end{remark}
\subsection*{ Proof of \Cref{genustheory_for_biquadratic}}

Let $p$ be a prime dividing the discriminant $D_K$ of the biquadratic field $K = \mathbb{Q}(\sqrt{d_1}, \sqrt{d_2})$. Then $p$ is ramified in $K$. By \Cref{ramification}, the ramification index satisfies $e_p \leq 2$. 

By genus theory, we have
\[
G(K) = \mathbb{Q}\left( \left\{ \sqrt{(-1)^{\frac{p-1}{2}}p} \;\middle|\; p \mid d_1d_2 \right\} \right)\cup k_2 ,
\]
where $k_2$ is a quadratic extension of $\mathbb{Q}$ accounting for the contribution of the prime $2$.

From genus theory for quadratic fields, we know that $k_2$ is either $\mathbb{Q}$ or $\mathbb{Q}(\sqrt{-1})$, depending on whether $e_2 = 1$ or $2$. Therefore, we can write:

\begin{align*}
G(K) =
\begin{cases}
\mathbb{Q}\big( \sqrt{p_1}, \dots, \sqrt{p_t}, \sqrt{-p_{t+1}}, \dots, \sqrt{-p_{t+s}}, \\
\qquad\; \sqrt{q_1}, \dots, \sqrt{q_\ell}, \sqrt{-q_{\ell+1}}, \dots, \sqrt{-q_{\ell+r}} \big), & \text{if } d_1 \equiv d_2 \equiv 1 \pmod{4}, \\[1ex]
\mathbb{Q}\big( \sqrt{p_1}, \dots, \sqrt{p_t}, \sqrt{-p_{t+1}}, \dots, \sqrt{-p_{t+s}}, \\
\qquad\; \sqrt{q_1}, \dots, \sqrt{q_\ell}, \sqrt{-q_{\ell+1}}, \dots, \sqrt{-q_{\ell+r}}, \sqrt{-1} \big), & \text{otherwise}.
\end{cases}
\end{align*}
\qed

	\begin{theorem} \label{part_Hilbert_classfield_for_biquadratic}
		Let $K = \mathbb{Q}(\sqrt{d_1}, \sqrt{d_2})$ be a biquadratic field and, where
		\[
		d_1 = p_1 \cdots p_t p_{t+1} \cdots p_{t+s}, \quad d_2 = q_1 \cdots q_\ell q_{\ell+1} \cdots q_{\ell+r}.
		\]
		Suppose  the following conditions are satisfied:
		\begin{itemize}
			\item $p_i \equiv 1 \pmod{4}$ for $1 \leq i \leq t$,
			\item $p_{t+j} \equiv 3 \pmod{4}$ for $1 \leq j \leq s$,
			\item $q_i \equiv 1 \pmod{4}$ for $1 \leq i \leq \ell$,
			\item $q_{\ell+j} \equiv 3 \pmod{4}$ for $1 \leq j \leq r$.
		\end{itemize}
		Let 
\begin{equation}\label{hilbertclassfld} 
\begin{aligned}
L: =
\begin{cases}
\mathbb{Q}\big( \sqrt{p_1}, \dots, \sqrt{p_t}, \sqrt{p_{t+1} p_{t+2}}, \dots, \sqrt{p_{t+1}p_{t+s}}, \\
\quad \sqrt{q_1}, \dots, \sqrt{q_\ell}, \sqrt{p_{t+1}q_{\ell+1}}, \dots, \sqrt{p_{t+1}q_{\ell+r}} \big), & \text{if } d_1 \equiv d_2 \equiv 1 \pmod{4}, \\[1ex]
\mathbb{Q}\big( \sqrt{p_1}, \dots, \sqrt{p_t}, \sqrt{p_{t+1}}, \dots, \sqrt{p_{t+s}}, \\
\quad \sqrt{q_1}, \dots, \sqrt{q_\ell}, \sqrt{q_{\ell+1}}, \dots, \sqrt{q_{\ell+r}},  \big), & \text{otherwise}.
\end{cases}
\end{aligned}
\end{equation}
If $2\nmid d_1d_2$, then $L \subseteq H(K)$. Moreover,
\[
(\Z/2\Z)^{\omega(\cD_K)-3} \subseteq \Cl_K.
\]
\end{theorem}

\begin{proof}
By \Cref{genustheory_for_biquadratic}, we have $L \subseteq G(K)$, where $G(K)$ is the genus field of $K$.  
The field $L$ is the maximal totally real subfield of $G(K)$.  
Hence $L/K$ is unramified at all primes, including the infinite ones.  
Therefore $L\subseteq H(K)$, and the inclusion 
$(\Z/2\Z)^{\omega(\cD_K)-3}\subseteq\Cl_K$  as $\Gal(L/K)\cong (\Z/2\Z)^{\omega(\cD_K)-3} \subset \Cl_K $.
This completes the proof. \qedhere
\end{proof}

\begin{cor}\label{Hilbert_classfield_for_biquadratic}
 Let $K=\Q(\sqrt{d_1},\sqrt{d_2})$ be a biquadratic field with  $2\nmid d_1d_2$. If  its Hilbert class field $H(K)$ of $K$ is abelian over $\Q$, then $H(K)=L$. In  other words, $H(K)$ is a multiquadratic field.
\end{cor}

\begin{proof}
Since $H(K)$ is abelian over $\Q$, it must lie inside the genus field $G(K)$ of $K$.
From the preceding theorem, we know that $L\subseteq H(K)$ and that $[G(K):L]\le2$.
Moreover, $L$ is the maximal subextension of $G(K)$ such that $L/K$ is unramified
at all infinite primes. Therefore, $H(K)$, being unramified over $K$ and contained
in $G(K)$, must coincide with $L$. Hence $H(K)=L$.
\end{proof}

\begin{cor}\label{cyclic:omega_result}
Let $K=\Q(\sqrt{d_1},\sqrt{d_2})$ be a biquadratic field with $2\nmid d_1d_2$.
If $\Cl_K$ is cyclic then $\omega(\Delta_K)\le4$.
\end{cor}

\begin{proof}
Genus theory provides
\[
(\Z/2\Z)^{\,\omega(\Delta_K)-3}\subseteq\Cl_K.
\]
A cyclic group cannot contain $(\Z/2\Z)^r$ for $r\ge2$, hence
$\omega(\Delta_K)-3\le1$.  Therfore  $\omega(\Delta_K)\le4$.
\end{proof}

	\begin{theorem}\label{Biquadrati:cyclicclass}
		Let $K$ be a biquadratic field whose Hilbert class field $H(K)$ is abelian over $\mathbb{Q}$. If the class group $\mathrm{Cl}_K$ is cyclic, then the class number $h_K$ is either $1$ or $2$.
	\end{theorem}
	
	\begin{proof}
		Let $K = \mathbb{Q}(\sqrt{d_1}, \sqrt{d_2})$ be a biquadratic field. By \Cref{Hilbert_classfield_for_biquadratic}, if the Hilbert class field $H(K)$ is abelian over $\mathbb{Q}$, then $H(K)$ must be a multiquadratic extension of $\mathbb{Q}$. In particular, $H(K)/\mathbb{Q}$ is a Galois extension with an abelian Galois group.

        \noindent
		Now, suppose the class group $\mathrm{Cl}_K$ is cyclic. Since $\Gal(H(K)/K) \cong \mathrm{Cl}_K,$
		we have that $\mathrm{Gal}(H(K)/K)$ is also cyclic. But a cyclic subgroup of a multiquadratic Galois group has order at most $2$, because multiquadratic extensions have Galois groups isomorphic to a product of copies of $\mathbb{Z}/2\mathbb{Z}$. Therefore,
		\[
		[H(K) : K] = |\mathrm{Cl}_K| \leq 2.
		\]
		Hence, the class number $h_K = |\mathrm{Cl}_K|$ satisfies $h_K \in \{1, 2\}$.
	\end{proof}
	
    \begin{prop}\label{structure:biquadratic}
Let $K$ be a biquadratic number field such that its Hilbert class field $H(K)$ is abelian over $K$, and suppose that the class group $\mathrm{Cl}_K$ is cyclic. Then $K$ must be of one of the following forms:
\begin{equation}\label{classicification:biquadratic}
\begin{aligned}
  & K = \mathbb{Q}\left(\sqrt{p_1}, \sqrt{p_2}\right), \\
  & K = \mathbb{Q}\left(\sqrt{p_1p_2}, \sqrt{p_2}\right), \\
  & K = \mathbb{Q}\left(\sqrt{p_1p_2}, \sqrt{p_3p_4}\right), \\
  & K = \mathbb{Q}\left(\sqrt{p_1p_2p_3}, \sqrt{p_4}\right)
\end{aligned}
\end{equation}
where the $p_i$ are distinct odd primes.
\end{prop}

\begin{proof}
Let $K = \mathbb{Q}(\sqrt{d_1}, \sqrt{d_2})$ be a biquadratic field, where $d_1$ and $d_2$ are square-free integers. The Hilbert class field $H(K)$ is the maximal unramified abelian extension of $K$, and by assumption, the class group $\mathrm{Cl}_K$ is cyclic. It follows that $\mathrm{Gal}(H(K)/K)$ is cyclic of order at most $2$, so $[H(K):K] \leq 2$.

Since $[K:\mathbb{Q}] = 4$, we have $[H(K):\mathbb{Q}] \leq 8$. On the other hand, $H(K) \subseteq G(K)$, where $G(K)$ denotes the genus field of $K$. It is known that $[G(K):\mathbb{Q}] = 2^{\omega(D_K)}$ and $[G(K):K]\leq 2$, where $\omega(D_K)$ is the number of distinct prime divisors of the discriminant $D_K$ of $K$.
As $[G(K):\mathbb{Q}] \leq 16$ and $H(K) \subseteq G(K)$, we conclude that $\omega(D_K) \leq 4$. Therefore, $K$ can involve at most four distinct odd primes in its defining square roots, which restricts the possible forms of $K$ to the listed cases.
\end{proof}

	%\section{ Galios group of Hilbert class fields of $\mathbb{Q}(\sqrt{q}, \sqrt{rs})$}\label{sec:structure}
\section[Galois group of Hilbert class fields of Q(sqrt{q}, sqrt{rs})]{Galois group of Hilbert class fields of $\mathbb{Q}(\sqrt{q}, \sqrt{rs})$}\label{sec:structure}

Let $K = \mathbb{Q}(\sqrt{q}, \sqrt{rs})$ be a biquadratic number field with class number $h_K = 2$, and suppose its Hilbert class field $H(K)/\mathbb{Q}$ is abelian. Consider the three quadratic subfields:
\[
K_1 = \mathbb{Q}(\sqrt{q}), \quad K_2 = \mathbb{Q}(\sqrt{rs}), \quad K_3 = \mathbb{Q}(\sqrt{qrs}).
\]
Since $H(K)/\mathbb{Q}$ is abelian, it follows that each extension $H(K_i)/\mathbb{Q}$ is also abelian. Using genus theory, we can compute the Hilbert class fields $H(K_i)$ explicitly.

Let $h_i = \#\mathrm{Cl}_{K_i}$ denote the class number of $K_i$. In particular, we focus on $K_1 = \mathbb{Q}(\sqrt{q})$.

\subsection*{Case 1: Either $q \equiv 3 \pmod{4}$, or at least one of $r$ or $s$ satisfies $r \equiv 1$ or $s \equiv 1 \pmod{4}$}

This is the simpler case. By results from \cite{KrPa25}, one can show that
\[
\mathbb{Q}(\sqrt{q}, \sqrt{r}, \sqrt{s}) \subseteq H(K).
\]
Since the class group of $K$ has order $2$, this implies that the above extension must be equal to $H(K)$.

\subsection*[Case 2: $q \equiv 1 \pmod{4}$ and both $r \equiv s \equiv 3 \pmod{4}$]{Case 2: $q \equiv 1 \pmod{4}$ and both $r \equiv s \equiv 3 \pmod{4}$}

In this case, we show that $H(K)/\mathbb{Q}$ cannot be abelian. Suppose, for contradiction, that $H(K)/\mathbb{Q}$ is abelian. Then so are the extensions $H(K_i)/\mathbb{Q}$ for each $i$.
By \Cref{Hilbert_classfield_for_biquadratic}, Hilbert class fields has to be $\mathbb{Q}(\sqrt{q},\sqrt{rs})$. Therefore, it can't have class number two. 
Hence, $H(K)/\mathbb{Q}$ is not abelian.

\begin{prop}
Let $K = \mathbb{Q}(\sqrt{q}, \sqrt{rs})$ be a biquadratic field with class number $2$. If the extension $H(K)/\mathbb{Q}$ is not abelian, then
\[
\mathrm{Gal}(H(K)/\mathbb{Q}) \cong D_8.
\]
\end{prop}

\begin{proof}
Since $H(K)/\mathbb{Q}$ is not abelian, the Galois group $\mathrm{Gal}(H(K)/\mathbb{Q})$ cannot be isomorphic to any abelian group of order $8$, such as $\mathbb{Z}/8\mathbb{Z}$, $\mathbb{Z}/4\mathbb{Z} \times \mathbb{Z}/2\mathbb{Z}$, or $(\mathbb{Z}/2\mathbb{Z})^3$.
The only nonabelian groups of order $8$ are the dihedral group $D_8$ and the quaternion group $Q_8$. However, $Q_8$ is not isomorphic to a Galois group of a number field extension over $\mathbb{Q}$, since it is not a transitive group, whereas Galois groups of field extensions are always transitive permutation groups.  Therefore, $\mathrm{Gal}(H(K)/\mathbb{Q})$ must be isomorphic to $D_8$.
\end{proof}

\begin{prop}
Let $K = \mathbb{Q}(\sqrt{q}, \sqrt{rs})$ be a biquadratic field with class number $2$. Suppose $q \equiv 1 \pmod{4}$ and $r, s \equiv 3 \pmod{4}$ are distinct odd primes. Then the extension $H(K)/\mathbb{Q}$ is not abelian. Moreover, both quadratic fields $\mathbb{Q}(\sqrt{q})$ and $\mathbb{Q}(\sqrt{rs})$ have class number one.
\end{prop}

\begin{proof}
Under the given congruence conditions, it follows from Galois-theoretic considerations that $\mathrm{Gal}(H(K)/\mathbb{Q}) \cong D_8$, the dihedral group of order $8$, which is nonabelian. This implies that $H(K)$ cannot be written as a compositum of quadratic fields, and in particular, it does not contain the full multiquadratic field $\mathbb{Q}(\sqrt{q}, \sqrt{r}, \sqrt{s})$.
Since $K$ has class number $2$, its Hilbert class field satisfies $[H(K) : K] = 2$, and since $[K : \mathbb{Q}] = 4$, we conclude that $[H(K) : \mathbb{Q}] = 8$.

Suppose, for contradiction, that either $\mathbb{Q}(\sqrt{q})$ or $\mathbb{Q}(\sqrt{rs})$ has class number greater than one. Then its Hilbert class field would be a nontrivial unramified abelian extension of degree at least $2$. Since $K$ contains both of these fields, the Hilbert class field of either would be embedded in $H(K)$. Also Hilbert class field of $\mathbb{Q}(\sqrt{q})$ or $\mathbb{Q}(\sqrt{rs})$ does not contain $K$.

But this would imply the existence of a normal extension of degree $4$ inside $H(K)$ corresponding to one of the quadratic fields. Thus, $H(K)$ would contain two distinct degree-$4$ normal subfields, which would contradict the structure of the Galois group $\mathrm{Gal}(H(K)/\mathbb{Q}) \cong D_8$, as $D_8$ has a unique normal subgroup of index $2$ and at most one normal subgroup of index $4$.
Therefore, both $\mathbb{Q}(\sqrt{q})$ and $\mathbb{Q}(\sqrt{rs})$ must have class number one.
\end{proof}

\begin{remark}
The fact that both $\mathbb{Q}(\sqrt{q})$ and $\mathbb{Q}(\sqrt{rs})$ have class number $1$ connects naturally to the classical class number one problem, which seeks a complete classification of quadratic fields with trivial class group. In the imaginary case, this classification is complete, due to foundational work by Heegner, Stark \cite{Sta66}, and Baker\cite{Bak71}. However, the real quadratic case remains unresolved and is deeply intertwined with difficult questions concerning fundamental units and regulators. This example illustrates how genus theory and Hilbert class fields can shed light on why certain quadratic fields have class number one.
\end{remark}

\begin{remark}
The class-number formula  \begin{align*}
    h_K=\frac{h_1h_2h_3\,[\mathcal{O}_K^\times:\mathcal{O}_{K_1}^\times\mathcal{O}_{K_2}^\times\mathcal{O}_{K_3}^\times]}{4}
\end{align*}
expresses $h_K$ in terms of quadratic class numbers and a unit index, but does
not uniquely determine the arithmetic of $K$ since the unit index can be
$1,2,4,$ or $8$ (see \cite{DuKi23}). In contrast, knowledge of the Galois group $\mathrm{Gal}(H(K)/\mathbb{Q})$ imposes strong structural constraints on the Hilbert class fields of $K$ and its subfields. For instance, if $\mathrm{Gal}(H(K)/\mathbb{Q}) \cong D_8$, then $H(K)/\mathbb{Q}$ is nonabelian and contains a unique subfield of degree $4$, namely $K$, which cannot be decomposed as a compositum of Hilbert class fields of quadratic subfields. This structural condition forces both $\mathbb{Q}(\sqrt{q})$ and $\mathbb{Q}(\sqrt{rs})$ to have class number $1$, to maintain consistency with the known value $h_K = 2$. Thus, the Galois-theoretic perspective rules out otherwise plausible class number combinations that the class number formula alone cannot distinguish.
\end{remark}

 A fundamental distinction arises between the genus theory of quadratic and that of biquadratic fields.
For quadratic fields, genus theory provides a complete description of the $2$-torsion subgroup of the class group.
In contrast, the situation for biquadratic fields is more intricate: genus theory does not, in general, account for the entire $2$-part of the class group.
Indeed, several explicit examples demonstrate that the genus field fails to capture all elements of order two in $\Cl(K)$.

\noindent
The following table lists several biquadratic fields of the form   $K = \mathbb{Q}(\sqrt{q}, \sqrt{rs}),$
where $ q \equiv 1 \pmod{4} $ and $ r, s \equiv 3 \pmod{4} $ are distinct prime numbers.  
For each field $ K $, the table also records the class numbers of its three quadratic subfields, namely  
\[
K_1 = \mathbb{Q}(\sqrt{q}), \qquad  
K_2 = \mathbb{Q}(\sqrt{rs}), \qquad  
K_3 = \mathbb{Q}(\sqrt{qrs}).
\]
In all examples considered, the subfields $ K_1 $ and $ K_2 $ have class number $ 1 $, whereas $ K_3 $ has class number $ 4 $.

\begin{center}
\begin{tabular}{|c|c|c|c|c|c|c|}
\hline 
$q$ & $r$ & $s$ & $h_ K$  & $h_{K_1}$ & $h_{K_2}$ & $h_{K_3}$ \\
\hline
5  & 11 & 19 & 2 & 1 & 1 & 4 \\
5  & 19 & 11 & 2 & 1 & 1 & 4 \\
5  & 19 & 31 & 2 & 1 & 1 & 4 \\
5  & 31 & 19 & 2 & 1 & 1 & 4 \\
13 & 3  & 23 & 2 & 1 & 1 & 4 \\
13 & 23 & 3  & 2 & 1 & 1 & 4 \\
37 & 3  & 7  & 2 & 1 & 1 & 4 \\
37 & 3  & 11 & 2 & 1 & 1 & 4 \\
37 & 7  & 3  & 2 & 1 & 1 & 4 \\
37 & 7  & 11 & 2 & 1 & 1 & 4 \\
37 & 11 & 3  & 2 & 1 & 1 & 4 \\
37 & 11 & 7  & 2 & 1 & 1 & 4 \\
\hline
\end{tabular}
\end{center}

\section{Euclidean Ideals in Biquadratic Fields}\label{sec:euclidean}

To investigate Euclidean ideals in biquadratic fields of the form 
$\Q(\sqrt{p_1},\sqrt{p_1p_2})$, the author, together with Krishnamoorthy
\cite{KrPa25}, established unconditional criteria for the existence of 
Euclidean ideal classes in certain real quartic extensions. 
We recall the relevant result here.

\begin{theoremquartic}[{\cite{KrPa25}}]\label{thm:Quatric}
Let $K$ be a real quartic extension of $\Q$ with prime class number, and suppose
that its Hilbert class field $H(K)$ is abelian over $\Q$.  
Let $f$ be the least common multiple of $16$ and the conductor of $K$.  
Write 
\[
G=\Gal(\Q(\zeta_f)/K),
\qquad
G_\ell=\Gal(\Q(\zeta_f)/\Q(\zeta_\ell))
\]
for each odd prime $\ell\mid f$ and for $\ell=4$.  
If
\[
G\not\subseteq \bigcup_{\ell} G_\ell \;\cup\; \Gal(\Q(\zeta_f)/H(K)),
\]
then $K$ contains a Euclidean ideal class.
\end{theoremquartic}

\medskip

Now let $K$ be a biquadratic field whose Hilbert class field $H(K)$ is 
abelian over $\Q$, and suppose that the class group $\Cl_K$ is cyclic.
By the classification in \Cref{classicification:biquadratic}, such fields
belong to one of a small number of explicitly described families.

If $h_K=1$, then by Theorem~2 of \cite{Gra11} together with the growth
condition of \cite{Gra13}, the field $K$ necessarily admits a Euclidean
ideal.  
Thus the remaining case of interest is when $K$ is an abelian biquadratic
field of \emph{class number two}.  
Such fields must be of one of the following forms:
\begin{enumerate}
    \item $K=\Q(\sqrt{p_1p_2},\sqrt{p_2})$, except in the exceptional case
    $p_1\equiv1\pmod4$ and $p_2,p_3\equiv3\pmod4$; \label{prime21}
    \item $K=\Q(\sqrt{p_1p_2},\sqrt{p_3p_4})$, where $p_3,p_4\equiv3\pmod4$; 
    \label{prime22}
    \item $K=\Q(\sqrt{p_1p_2p_3},\sqrt{p_4})$, where $p_1,p_2\equiv3\pmod4$. 
    \label{primes31}
\end{enumerate}

To show that such a field $K$ contains a Euclidean ideal, it suffices via  Theorem A to find an element
\[
\sigma\in\Gal(H(K)/K)
\]
whose lift $\hat{\sigma}\in\Gal(\Q(\zeta_f)/K)$ avoids the exceptional union
\[
\bigcup_{\ell}G_\ell\;\cup\;\Gal(\Q(\zeta_f)/H(K)).
\]
For fields of type \cref{prime21}, such a lift is established in
\cite[Lemma~4.2]{KrPa25}.  
We now treat the remaining two cases.

\begin{lemma}\label{existence_of_lift3}
Let $p_i$ be distinct primes with $p_3,p_4\equiv3\pmod4$, and let
\[
K=\Q(\sqrt{p_1p_2},\sqrt{p_3p_4})
\]
be a biquadratic field of class number two.  
If $\Gal(H(K)/K)=\langle\sigma\rangle$, then there exists a lift
$\hat{\sigma}\in\Gal(\Q(\zeta_f)/K)$ of $\sigma$ such that
\[
\hat{\sigma}\notin 
\bigcup_{\ell} \Gal(\Q(\zeta_f)/\Q(\zeta_\ell))
\;\cup\;
\Gal(\Q(\zeta_f)/H(K)),
\]
where the union runs over all odd primes $\ell\mid f$ and $\ell=4$.
\end{lemma}

\begin{proof}
The nontrivial element $\sigma\in\Gal(H(K)/K)\simeq\Z/2\Z$ must act
nontrivially on $H(K)$, so any lift $\hat{\sigma}$ necessarily avoids
$\Gal(\Q(\zeta_f)/H(K))$.

Because $H(K)$ is generated by adjoining square roots of suitable products of
the primes $p_i$, and since $p_3,p_4\equiv3\pmod4$, at least one such square
root (e.g.\ $\sqrt{p_1p_3}$ or $\sqrt{p_1p_4}$) lies in 
extensions whose interaction with $\Q(\zeta_\ell)$ ensures that 
$H(K)\cap\Q(\zeta_\ell)\subsetneq H(K)$ for some $\ell$.

If $H(K)\cap\Q(\zeta_\ell)\not\subset K$, then every lift of $\sigma$ avoids
$\Gal(\Q(\zeta_f)/\Q(\zeta_\ell))$.  
If, on the other hand, $H(K)\cap\Q(\zeta_\ell)\subseteq K$, one may extend
$\sigma$ to the compositum $H(K)\Q(\zeta_\ell)$ so that it restricts to
$\sigma$ on $H(K)$ and acts nontrivially on $\Q(\zeta_\ell)$.  
In either case, we obtain a lift $\hat{\sigma}$ outside the exceptional
union.
\end{proof}

The final remaining case is structurally similar to the situation treated in 
\Cref{existence_of_lift3}.  
The same Galois-theoretic considerations apply, and the argument establishing the
existence of an appropriate lift carries over verbatim.  
For completeness we record the statement below, but omit the proof.
\begin{lemma}\label{existence_of_lift4}
Let $p_i$ be distinct primes with $p_1,p_2\equiv3\pmod4$, and let
\[
K=\Q(\sqrt{p_1p_2p_3},\sqrt{p_4})
\]
be a biquadratic field of class number two.  
If $\Gal(H(K)/K)=\langle\sigma\rangle$, then there exists a lift
$\hat{\sigma}\in\Gal(\Q(\zeta_f)/K)$ of $\sigma$ satisfying the same
avoidance property as in \Cref{existence_of_lift3}.
\end{lemma}
\begin{proof}
    Proof as same as  above lemma. 
\end{proof}

\begin{theorem}
Let $K=\Q(\sqrt{d_1},\sqrt{d_2})$ be a biquadratic field with 
$2\nmid d_1d_2$, and suppose that its Hilbert class field $H(K)$ is abelian
over $\Q$.  
Then $\Cl_K$ is cyclic if and only if $K$ contains a Euclidean ideal class.
\end{theorem}

\begin{proof}
Assume first that $\Cl_K$ is cyclic.  
Then by \Cref{Biquadrati:cyclicclass}, $h_K\le2$.  
If $h_K=1$, we are done by \cite{Gra11,Gra13}.  
If $h_K=2$, then $K$ must belong to one of the three families listed above,
and in each case \Cref{existence_of_lift3,existence_of_lift4} (together with 
\cite{KrPa25}) produces a lift $\hat{\sigma}$ satisfying the hypothesis of 
\Cref{thm:Quatric}.  
Hence, $K$ possesses a Euclidean ideal. Conversely, if $K$ contains a Euclidean ideal class, then 
\cite[Theorem~1.6]{Len79} implies that $\Cl_K$ is cyclic.
\end{proof}

\section{Density of Biquadratic Fields with Euclidean Ideals}

Before addressing the density problem, we record the following useful observation.

\begin{prop}
Let $K=\Q(\sqrt{d_1}, \sqrt{d_2})$ be a biquadratic number field such that $2\nmid d_1 d_2$.
If $K$ possesses a Euclidean ideal, then $\omega(\cD_K)\le4$.
Equivalently, the genus number of $K$ satisfies $g_K\le 2^2$.
\end{prop}

\begin{proof}
Assume that $K$ admits a Euclidean ideal.  
By \cite[Theorem~1.6]{Len79}, any number field possessing a Euclidean ideal has a cyclic class group, i.e.\ $\Cl_K$ is cyclic.  
Invoking \Cref{cyclic:omega_result}, we then obtain that $\omega(\cD_K)\le4$ for such a field $K$. 
In particular, the number of distinct prime discriminant divisors of $K$ is at most four.
Since the genus number satisfies $g_K = 2^{\omega(\cD_K)-2}$, it follows that $g_K \le 2^{4-2} = 2^2 $. Hence, every biquadratic field with a Euclidean ideal has genus number at most $2^2$, as required.
\end{proof}

\noindent
We now turn to the task of counting biquadratic fields according to their
genus numbers. Our aim is to describe the distribution of genus numbers
within this family and to obtain asymptotic formulas for the number of
biquadratic fields whose genus number is prescribed or lies within a specified range.
Recent years have witnessed substantial progress on genus statistics for other
families of number fields. For instance, McGown and Tucker~\cite{McTu23}
established genus distribution results for cubic fields, and together with
Throne~\cite{MTT23} they extended these methods to the quintic setting.
Dixit and Pasupulati~\cite{DiPa26} have recently obtained analogous counts
for $S_3\times C_q$ and $D_4$ extensions. In the abelian case, Frei, Loughran,
and Newton~\cite{FLR23} counted fields ordered by conductor and proved that
the density of abelian fields with a prescribed Galois group and genus number
is zero.

Passing from  counts by conductor to discriminant, however, is delicate and only
straightforward in prime-degree extensions. In contrast, our treatment of
biquadratic fields proceed more directly: We make a slight modification to clever 
parameterization due to Rome\cite{Rom18} and precise counting arguments to obtain the desired genus
statistics.

\subsection{Parametrization of Biquadratic Fields}
 To begin, we introduce a
uniform parametrization of the family of biquadratic fields that will be used
throughout the counting arguments.
Let $ K = \mathbb{Q}\left(\sqrt{a}, \sqrt{b}\right) $. 
Writing $ a = m a_1 $, $ b = m b_1 $ with pairwise coprime, squarefree 
$ m, a_1, b_1 $, we have
\begin{align*}
	K = \mathbb{Q}\left(\sqrt{m a_1}, \sqrt{m b_1}\right),
\end{align*}
whose quadratic subfields are
\[
\mathbb{Q}\left(\sqrt{m a_1}\right),\quad \mathbb{Q}\left(\sqrt{m b_1}\right),\quad
\mathbb{Q}\left(\sqrt{a_1 b_1}\right).
\]
Thus, $ K $ is uniquely determined (up to permutation) by the triple 
$\left(m_1, m_2, m_3\right) = \left(m, a_1, b_1\right)$. The discriminant takes the form
\[
\Delta_K = c^2 m_1^2 m_2^2 m_3^2,
\]
with $ c \in \{1,4,8\} $ depending only on $ m_i \bmod 4 $. 
If all $ m_i $ are odd, then $ c=8 $ never occurs, and
\[
c = 
\begin{cases}
	1, & m_1 \equiv m_2 \equiv m_3 \pmod{4},\\
	4, & \text{otherwise.}
\end{cases}
\]
\subsection{Counting biquadratic fields in $\cS$}

For $X>0$ put
\[
\cS(X)=\# \{K\in\cS:\;|\Delta_K|\le X\}.
\]

\begin{theorem}
	
	As $X\to\infty$,
	\[
	\cS(X)=\frac{7}{1920}\,\sqrt{X}\,\log^2 X\;\prod_{p}\Big(1-\frac{1}{p}\Big)^3\Big(1+\frac{3}{p}\Big)
	+O(\sqrt{X}\log X).
	\]
\end{theorem}

\begin{proof}
	We count fields $K\in\cS$ using the parametrization above. 
	The condition $\Delta_K \le X$ is equivalent to
	\[
	m_1 m_2 m_3 \le \frac{\sqrt{X}}{c},
	\]
	where $c$ depends only on the residue classes of the $m_i$ modulo $4$. 
	Let $\epsilon=(\epsilon_1,\epsilon_2,\epsilon_3)\in\{\pm1\}^3$ denote these residues, 
	and define
	\[
	T(\epsilon;N)
	= \#\{(m_1,m_2,m_3): m_i\ \text{odd, squarefree, coprime},\
	m_i\equiv\epsilon_i\pmod4,\
	m_1m_2m_3\le N\},
	\]
	where $N_\epsilon=\sqrt{X}/c_\epsilon$ and $c_\epsilon\in\{1,4\}$.
		In this definition, each $m_i$ is restricted to a fixed congruence class modulo $4$. 
	To handle these restrictions analytically, we use the non-principal Dirichlet character modulo $4$, defined by
	
	\[
	\chi(n)=
	\begin{cases}
		0,&2\mid n,\\
		1,&n\equiv1\pmod4,\\
		-1,&n\equiv3\pmod4.
	\end{cases}
	\]
	Then
	\[
	1_{m_i\equiv\epsilon_i\pmod4}
	=\frac{1}{2}\sum_{\nu\in\{1,2\}}\chi^\nu(m_i\epsilon_i).
	\]
	Applying this to all three $m_i$ gives
\begin{align*}
		T(\epsilon;N)&
		= \sum_{\substack{m_1m_2m_3 \le N \\ (m_1,m_2,m_3)=1}}
		\mu^2(m_1m_2m_3)
		\prod_{i=1}^{3} 1_{\,m_i \equiv \epsilon_i \pmod{4}}.\\
&	= \frac{1}{8}\!\!\!\sum_{\nu_1,\nu_2,\nu_3\in\{1,2\}}
	\sum_{\substack{m_1m_2m_3\le N \\ (m_1,m_2,m_3)=1}}
	\mu^2(m_1m_2m_3)\,
	\chi^{\nu_1}(\epsilon_1 m_1)\chi^{\nu_2}(\epsilon_2 m_2)\chi^{ \nu_3}(\epsilon_3 m_3).\\
	&= \frac{1}{8}\sum_{\nu_1,\nu_2,\nu_3\in\{1,2\}}
	C(\nu,\epsilon)\;
	S_{\nu_1,\nu_2,\nu_3}(N),
\end{align*}
where
\[
C(\nu,\epsilon):=\prod_{i=1}^3\chi^{\nu_i}(\epsilon_i)\in\{\pm1\},
\qquad
S_{\nu_1,\nu_2,\nu_3}(N)
:=\sum_{\substack{m_1m_2m_3\le N\$m_1,m_2,m_3)=1}}
\mu^2(m_1m_2m_3)\prod_{i=1}^3\chi^{\nu_i}(m_i).
\]
	where the factor $\mu^2$ enforces squarefreeness.  This representation makes the
sums multiplicative and prepares them for analysis via the Dirichlet series and the
Selberg--Delange method.

	\noindent	After expressing the congruence conditions in terms of characters,
the counting problem reduces to estimating sums of the form
\[
S(N;\nu_1,\nu_2,\nu_3)
= \sum_{\substack{m_1m_2m_3\le N \\ (m_1,m_2,m_3)=1}}
\mu^2(m_1m_2m_3)\,
\chi^{\nu_1}(m_1)\chi^{\nu_2}(m_2)\chi^{\nu_3}(m_3),
\]
where each $\nu_i\in\{1,2\}$ and $\chi$ is the non-principal character modulo $4$.

\noindent
Define
\[
a_n = \mu^2(n)\!\!\sum_{m_1 m_2 m_3 = n}
\chi^{\nu_1}(m_1)\chi^{\nu_2}(m_2)\chi^{\nu_3}(m_3), \qquad
S(N;\nu_1,\nu_2,\nu_3)=\sum_{n\le N}a_n.
\]
Its Dirichlet series
\[
F(s)=\sum_{n=1}^\infty\frac{a_n}{n^s},\qquad \Re(s)>1,
\]
is multiplicative and satisfies
\[
F(s)=\prod_{p}\Big(1+\frac{\chi^{\nu_1}(p)+\chi^{\nu_2}(p)+\chi^{\nu_3}(p)}{p^s}\Big)
=G(s)\prod_{i=1}^3L(\chi^{\nu_i},s),
\]
where
\[
G(s)=\prod_{p}\Bigg(1+\frac{\chi^{\nu_1}(p)+\chi^{\nu_2}(p)+\chi^{\nu_3}(p)}{p^s}\Bigg)
\prod_{i=1}^3\Big(1-\frac{\chi^{\nu_i}(p)}{p^s}\Big),
\]
which converges absolutely for $\Re(s)>\frac{1}{2}$.

If all $\nu_i=2$, then the Dirichlet series $F(s)$ contains the factor
$\zeta(s)^3$ and therefore has a triple pole at $s=1$; for any other choice
of $(\nu_1,\nu_2,\nu_3)$ the pole order is at most~$2$, leading only to
$O(N\log N)$ contributions.  In the principal case $\nu_1=\nu_2=\nu_3=2$,
we may write
\[
F(s)=G(s)\,\zeta(s)^3,
\]
where $G(s)$ is holomorphic and nonvanishing in a neighborhood of $s=1$.
Hence, the Selberg--Delange theorem
(\cite[Theorem~II.5.3]{Ten15book}) applies and gives
\[
\sum_{n\le N}a_n
= \frac{G(1)}{\Gamma(3)}\,N(\log N)^2 + O(N\log N)
= \frac{1}{2}G(1)\,N(\log N)^2 + O(N\log N).
\]
A direct evaluation of the Euler product for $G(1)$ shows that
\[
G(1)
= 
\prod_{p\neq 2}\Bigl(1-\frac{1}{p}\Bigr)^{3}
\Bigl(1+\frac{3}{p}\Bigr).
\]

For each residue pattern $\epsilon$,
\begin{align*}
T(\epsilon;N) 	&= \frac{1}{8}\sum_{\nu_1,\nu_2,\nu_3\in\{1,2\}}
C(\nu,\epsilon)\;
S_{\nu_1,\nu_2,\nu_3}(N),\\
&= A\,N(\log N)^2 + O(N\log N),
\end{align*}
%\qquad
%c_0=\frac{1}{5}\prod_{p}\Big(1-\frac{1}{p}\Big)^3\Big(1+\frac{3}{p}\Big).
Where,
\[
A=\frac{G(1)}{16}
=\frac{1}{16}\prod_{p\neq 2}\Big(1-\frac1p\Big)^3\Big(1+\frac3p\Big),
\]
Summing over all $\epsilon$: two patterns have $c_\epsilon=1$, six have
$c_\epsilon=4$, hence
\[
\sum_{\epsilon} T(\epsilon;N_\epsilon)
=2A\sqrt{X}\log^2(\sqrt{X})
+6A\frac{\sqrt{X}}{4}\log^2\!\Big(\frac{\sqrt{X}}{4}\Big)
+O(\sqrt{X}\log X).
\]
Since $\log^2(\sqrt{X}/c)=\tfrac14(\log X)^2+O(\log X)$,
the coefficient of $\sqrt{X}(\log X)^2$ equals $A\cdot\tfrac14(2+\tfrac64)=A\tfrac78$.
Dividing by $6$ for the permutations of $(m_1,m_2,m_3)$ gives
\[
\cS(X)
=\frac{7}{48} A \sqrt{X}(\log X)^2+O(\sqrt{X}\log X),
\]
and substituting $A$ yields
\[
\cS(X)
=\frac{7}{768}\sqrt{X}(\log X)^2
\prod_{p\neq 2}\Big(1-\frac1p\Big)^3\Big(1+\frac3p\Big)
+O(\sqrt{X}\log X).
\]

\end{proof}

\subsection{Counting biquadratic fields with a fixed genus number}
We count biquadratic fields $K=\mathbb{Q}(\sqrt a,\sqrt b)$ with $a,b$ odd squarefree, discriminant $\Delta_K\le X$, and with $\Delta_K$ having exactly $n$ distinct prime factors.

\begin{theorem}\label{thm:main}
	Fix an integer $n\ge1$. Let
	\[
	\cS_{n}(X)=\#\{K\in\cS:\  g_K=2^{n},\ |\cD_K|\le X\}.
	\]
	Then, as $X\to\infty$,
	\[
	\cS_{n}(X)=O \left(  \frac{\sqrt{X}}{\log X}\,(\log\log X)^{\,n-1} \right),
	\]
	where the implied constant may depend on $n$.
\end{theorem}

\begin{proof}
	Recall that for any $K\in\cS$ the genus number satisfies 
	\[
	g_K = 2^{\omega(\cD_K)-2}.
	\]
	Thus, the condition $g_K=2^n$ is equivalent to $\omega(\cD_K)=n+2$. We count such fields using the standard parametrization
	\[
	K = \Q(\sqrt{m_1m_2},\,\sqrt{m_1m_3}),
	\]
	where $m_1,m_2,m_3$ are pairwise coprime squarefree integers.  
	The discriminant of $K$ has the form
	\[
	\Delta_K = c^2\,(m_1m_2m_3)^2,
	\]
	with $c\in\{1,4\}$, and hence the condition $|\Delta_K|\le X$ is equivalent to
	\[
	M := m_1m_2m_3 \le N_\epsilon := \frac{\sqrt{X}}{c_\epsilon},
	\]
	where $c_\epsilon\in\{1,4\}$ depends on the residue pattern 
	$\epsilon=(\epsilon_1,\epsilon_2,\epsilon_3)\in\{\pm1\}^3$
	with $m_i\equiv\epsilon_i\pmod{4}$.  
	Since the odd prime divisors of $\Delta_K$ coincide with those of $M$,  
	the condition $\omega(\Delta_K)=n$ translates to $\omega(M)=n$
	(up to a harmless shift in certain congruence patterns).
	
	We now estimate the number of admissible ordered triples $(m_1,m_2,m_3)$.
	Each biquadratic field corresponds to at most $3!$ such ordered triples, arising from
	permutations of the three quadratic subfields, so it suffices to bound the total number of ordered triples satisfying these conditions.
	Fix a residue pattern $\epsilon\in\{\pm1\}^3$.  
	We consider ordered triples of odd, pairwise coprime, squarefree integers $m_1,m_2,m_3$
	such that
	\[
	m_i\equiv\epsilon_i\pmod{4}, \qquad M=m_1m_2m_3\le N_\epsilon, \qquad \omega(M)=n.
	\]
	For each squarefree $M$ with $\omega(M)=n$, the number of ordered coprime factorizations
	$M=m_1m_2m_3$ is at most $3^n$, since each of the $n$ distinct prime divisors of $M$
	may be independently assigned to one of the three factors.  
	Hence
	\begin{align*}
			\#\left\{(m_1,m_2,m_3) \big\vert \ m_i\equiv \epsilon_i\pmod{4},  \mu^2(m_1m_2m_3)=1,\omega(m_1m_2m_3)=n\right\}\\
		\le 3^n\cdot \#\{M\le N_\epsilon:\ \mu^2(M)=1,\ \omega(M)=n\}.
	\end{align*}

    \noindent
	There are eight residue patterns $\epsilon\in\{\pm1\}^3$, each corresponding to
	$N_\epsilon=\sqrt X/c_\epsilon$ with $c_\epsilon\in\{1,4\}$.  
	Therefore,
	\[
	\#\left\{(m_1,m_2,m_3) \big\vert \ \mu^2(m_1m_2m_3)=1,\omega(m_1m_2m_3)=n\right\}
	\le 8\cdot 3^n\max_{\epsilon}\#\{M\le N_\epsilon:\ \mu^2(M)=1,\ \omega(M)=n\}.
	\]
	Since $N_\epsilon\le\sqrt X$, it suffices to estimate
	$\#\{M\le N:\mu^2(M)=1,\omega(M)=n\}$ for $N\asymp\sqrt X$. By the Sathe--Selberg theorem (\Cref{lem:ss}), for fixed $n\ge1$ and $N\to\infty$,
	\[
	\#\{m\le N:\mu^2(m)=1,\omega(m)=n\}
	\ll_n \frac{N}{\log N}(\log\log N)^{n-1}.
	\]
	Applying this with $N=\sqrt X$ gives
	\[
	\#\{M\le\sqrt X:\mu^2(M)=1,\omega(M)=n\}
	\ll_n \frac{\sqrt X}{\log\sqrt X}(\log\log\sqrt X)^{n-1}.
	\]
	Combining these estimates yields
\begin{align*}
	\#\left\{(m_1,m_2,m_3) \big\vert \,  \mu^2(m_1m_2m_3)=1,\omega(m_1m_2m_3)=n\right\}
\ll_n 3^n\frac{\sqrt X}{\log\sqrt X}(\log\log\sqrt X)^{n-1}.
\end{align*}
	Since each field corresponds to at most $6$ such triples, and all constant factors
	arising from the congruence expansions (for instance, the factor $2^{-3}$) are fixed,
	we obtain
\begin{align*}
		\cS_n(X)&\ll \#\left\{(m_1,m_2,m_3) \big\vert \,  \mu^2(m_1m_2m_3)=1,\omega(m_1m_2m_3)=n+2\right\} \\
	&\ll  n \frac{\sqrt X}{\log\sqrt X}(\log\log\sqrt X)^{n+1}.
\end{align*}
	Finally, using $\log\sqrt X=\tfrac12\log X$ and
	$\log\log\sqrt X=\log\log X + O(1)$, we arrive at the stated form
	\[
		\cS_n(X)\ll_n \frac{\sqrt X}{\log X}(\log\log X)^{n+11}.
	\]
	This completes the proof.
\end{proof}

\begin{cor}
	The set $\{K\in\cS:\  g_K=2^{n}\}$ has density zero in set $\cS$. 
\end{cor}

\subsection*{Proof of \Cref{density_Euclidean}}
\begin{proof}
	By the proposition proved earlier, if $K\in\mathcal S$ has a Euclidean ideal, then $g_K=2^2$.
	Hence
	\[
	\mathcal S_{\mathcal E}\subseteq \bigcup_{m\le 2}\{K\in\mathcal S:\ g_K=2^m\}.
	\]
	For each fixed $m$,  Theorem~\ref{thm:main} gives
	\[
	\cS_m(X)=\#\{K\in\mathcal S:\ \omega(\Delta_K)=m,\ |\Delta_K|\le X\}
	\ll_m \frac{\sqrt X}{\log X}(\log\log X)^{m-1},
	\]
	whereas the total count satisfies $\#\mathcal S(X)\asymp \sqrt X\log^2 X$.
	Dividing shows $\cS_m(X)/\#\mathcal S(X)\to0$ as $X\to\infty$. Since
	$\mathcal S_{\mathcal E}(X)$ is contained in the  set   $ \cup_{m\leq 2}  \cS_m(X) $, therefore  same conclusion holds for $\mathcal S_{\mathcal E}(X)$.
	Thus
	\[
	\lim_{X\to\infty}\frac{\#\mathcal S_{\mathcal E}(X)}{\#\mathcal S(X)}=0,
	\]
	as claimed.
\end{proof}

\section*{Acknowledgments}

The author gratefully acknowledges the Institute of Mathematical Sciences (IMSc), Chennai, and the Max Planck Institute for Mathematics (MPIM), Bonn, for providing excellent research facilities and a stimulating environment. 
I would also like to express my sincere gratitude to Karthick Babu, Anup Dixit, Sanoli Gun, Pieter Moree, and
 B. Sury for their valuable discussions and insightful suggestions throughout the development of this work.

    	\bibliographystyle{amsplain}
	\bibliography{ref.bib}
\end{document}